\documentclass[10pt]{article}
\usepackage{amsmath,amssymb,amsthm,graphics,graphicx,epsfig,mathtext}
\usepackage[cp1251]{inputenc}
\usepackage[T2A]{fontenc}
\usepackage[english]{babel}
\newtheorem{theorem}{Theorem}[section]
\newtheorem{corollary}{Corollary}

\newtheorem{lemma}[theorem]{Lemma}

\theoremstyle{definition}

\newtheorem{remark}{Remark}

\title{Asymptotic Stabilization of a Flexible Beam with an Attached Mass}

\author{Julia Kalosha\thanks{
Institute of Applied Mathematics and Mechanics, National Academy of Sciences of Ukraine, Sloviansk ({\tt\small  julykucher@gmail.com, alexander.zuyev@gmail.com})
        \newline
$^{**}$Otto von Guericke University Magdeburg, Germany\newline
$^{***}$Max Planck Institute for Dynamics of Complex Technical Systems, Magdeburg, Germany\newline
},  Alexander Zuyev$^{*,**,***}$}
\date{}

\begin{document}
\maketitle




\begin{abstract}
A mathematical model of a simply supported Euler--Bernoulli beam with attached spring-mass system is considered. The model is controlled by distributed piezo actuators and  lumped force. We address the issue of asymptotic behavior of solutions of this system driven by a linear feedback law. The precompactness of trajectories is established for the operator formulation of the closed-loop dynamics. Sufficient conditions for strong asymptotic stability of the trivial equilibrium are obtained.
\end{abstract}



\section{Introduction}

Mechanical structures with flexible beams are widely used in modern engineering, particularly in areas such as spacecraft manufacturing, large-scale robotics, wind turbine industry, and offshore drilling technology.
Along with rapid ongoing technological progress, the line between control engineering and mathematical control theory is blurring.
Inspired by industrial challenges, the control theory of elastic distributed parameter systems has been refined over the past several decades.
The overview of some important results in the field of stabilization of mechanical systems with flexible beams is presented in Section~\ref{sec_RelWork}.
This paper is devoted to the stabilization problem of a flexible beam model with $k$ distributed piezo actuators and attached mass (shaker).

Our study is motivated by the needs of rigorous theoretical treatment of an experimental setup considered in~\cite{DSK2014}. The current paper continues our previous works on the nonasymptotic stability analysis~\cite{ZK2013} and  eigenvalues distribution~\cite{KZB2021} of this class of systems.

The considered model consists of a flexible beam of the length $l$ and a controlled spring-mass system (shaker) attached at a point with the coordinate $l_0\in (0,l)$. The state of thia system is described by the cross-sectional deflection $w(x,t)$ of the beam centerline at a point $x\in(0,l)$ and time $t$; $E(x)>0$ and $I(x)>0$ are the Young modulus and the cross-section moment of inertia, respectively; $\rho(x)>0$ is the mass per unit length of the beam. A sketch of the considered mechatronic model is shown in Figure~\ref{pic_ShakerModel}.
\begin{figure}[htp]
\begin{center}
  \includegraphics[width=2in]{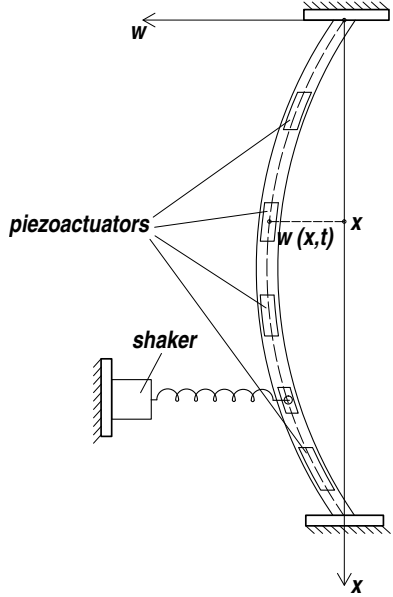}\\
  \caption{Sketch of the model}\label{pic_ShakerModel}
  \end{center}
\end{figure}

The distributed control is provided by $k$ piezoelectric actuators. The action of the $j$-th actuator is described by the torque density $M_j$ and the shape function $\psi_j(x)$ that satisfies the assumptions ${{\rm supp}\,\psi_j\cap\{0,l_0,l\}=\emptyset}$, $\psi_j''(x)\in C^2[0,l]$, $j=\overline{1,k}$. The lumped control is the force $F$ implemented by the electromagnetic shaker at $l_0\in(0,l)$; $m>0$ and $\varkappa>0$ are the mass of the moving part of the shaker and its stiffness, respectively. The dynamics of the system is described by the Euler--Bernoulli differential equation
\begin{equation}\label{Osc}
\rho(x)\ddot w(x,t)+E(x)I(x)w''''(x,t)=\sum\limits_{j=1}^k\psi_j''(x)M_j,\quad x\in(0,l)\setminus\{l_0\}
\end{equation}
subject to the boundary conditions
\begin{equation}\label{BC}
\begin{aligned}
&w(0,t)=w(l,t)=0,\\
&w''(0,t)=w''(l,t)=0
\end{aligned}
\end{equation}
and the interface condition
\begin{equation}\label{IC}
\left.\left(m\ddot w(x,t)+\varkappa w(x,t)\right)\right\rvert_{x=l_0}=E(x)I(x)\left(\left.w'''(x,t)\right\rvert_{x=l_0-0}-\left.w'''(x,t)\right\rvert_{x=l_0+0}\right)+F.
\end{equation}
Here the dot denotes time derivative,
and the prime stands for the derivative  with respect to $x$.

A feedback control that provides stability of the trivial equilibrium of the above system in Lyapunov's sense has been proposed in the paper~\cite{ZK2013}. The question about the asymptotic stability of this flexible system with attached mass remained open so far.
{\em The positive answer to this question is the main result of this paper that will be revealed in Theorem~\ref{th_2}.}

The structure of our paper is as follows. A survey of related results in the literature is summarized in Section~\ref{sec_RelWork}.
A special emphasis is put on distributed parameter systems described by wave and beam equations.
The operator representation of control system~\eqref{Osc}--\eqref{IC} is presented in Section~\ref{sec_DifOp} together with the necessary auxiliary results. The maximal invariant set is studied as well.
Our asymptotic stability study is based on an infinite-dimensional version of LaSalle's invariance principle~\cite{L1976} (cf.~\cite{Z2006})
where the precompactness of trajectories plays a crucial role. Thus, we present a detailed precompactness analysis in Section~\ref{swc_Precomp}.

The following contributions distinguish our results from the known ones in the literature:
\begin{itemize}
  \item explicit design of feedback controllers that ensure strong asymptotic stability of the closed-loop;
  \item asymptotic stability conditions with an arbitrary number of actuators;
  \item sufficient conditions for the precompactness of trajectories (Theorem~\ref{th_1}) allowing the localization of limit sets which are applicable for the case of actuators
  in the nodes of eigenfunctions.
\end{itemize}

\section{Related Work}\label{sec_RelWork}
The stability and control theory of dynamical systems described by hyperbolic differential equations has been developed by means of different approaches in various works since the second half of the 20th century.

In~\cite{R1967}, the nonharmonic Fourier theory is extended for second-order differential equations and the properties of eigenvalues of the moment problem are derived under basic assumptions on the control function.

The optimal control problems for the elastic beam vibrations are discussed in details in~\cite{K1972} for all the main kinds of boundary conditions derived from the mechanical set-up. Some general types of control, including impulse, are discussed for symmetric hyperbolic systems, while the case of the distributed load control is described for the vibrating beam model.

The generalization of the finite-dimensional linear control problem to an abstract linear control problem in Hilbert space is given in~\cite{K1992}. The conditions of controllability of vibrating strings and beams with distributed and boundary controls are presented in general fashion.

Since it is suitable to investigate the trajectories of complex mechanical systems in the form of properties of abstract differential equations in infinite-dimensional spaces, many results would not have been possible without the developed theory of $C_0$-semigroups, which is extensively covered in~\cite{LGM1999},~\cite{P2012} with applications to the stability problems.

Strong stability, stabilizability and detectability are studied by virtue of spectral methods in~\cite{O2000} for those control systems with bounded input and output operators that are not necessarily exponentially stable.

Some implementations of semigroups theory to the observability and controllability problems for abstract wave equations are reviewed in~\cite{KL2005}.

The authors of~\cite{DZ2006} brought universal approaches employing the non-harmonic Fourier series to the investigation of the controllability of networks of strings giving by that a useful tools for analysing the beam equations as well.

The inequalities of controllability of the linear partial differential equations are proved in~\cite{C2007}.

Semigroup theory and its application to the investigation of controllability, observability, stabilizability, exponential stabilizability and detectability of linear control systems with bounded input and output operators in infinite-dimensional spaces is covered in an integrated fashion in~\cite{CZ2012}.

Asymptotic stability properties of the boundary control Euler-Bernoulli beam connected to a nonlinear mass spring damper system is provided in~\cite{LZR2017} under weak assumptions on non-linear spring and damper.

A series of works by M.Shubov et al. (\cite{SS2016},~\cite{SK2018},~\cite{SK2019}) is dedicated to controllability and spectral analysis of Euler-Bernoulli beam models subject to different types of boundary conditions.

In recent years more and more mathematical interest is attracted by control problems arising in modelling of offshore drilling structures (\cite{LW2020},~\cite{LH2020}) and different kinds of robotic manipulators (\cite{OD2012},~\cite{WF2015}).

The stabilization problem for flexible-link manipulators with a payload is solved in~\cite{ZS2005} for the case of multiple passive joints
and in~\cite{ZS2007} for a manipulator represented by the Timoshenko beam model.

New results in the field of stabilization of infinite-dimensional dynamical systems in abstract spaces are presented in~\cite{SZ2021}.

Despite the advanced general approaches to the theory of stabilization and control of oscillating systems, many special cases have their peculiarities affecting the final stability results. For instance, the elastic structures consisting of flexible beams with distributed or boundary controls, multi-link networks of joint beams are widely studied in the context of controllability and stabilizability while the stability and asymptotic stability for the flexible beam with an attached rigid body and distributed control is the framework for investigation.

\section{Auxiliary Results}\label{sec_DifOp}
In order to treat the asymptotic stability property formally, we introduce an operator representation of the control system~\eqref{Osc}--\eqref{IC}.
For this purpose we consider the Hilbert space $X=\overset{\;\,\circ}H\,\!^2(0,l)\times L^2(0,l)\times \mathbb R^2$ equipped with the inner product
\begin{equation*}
\left\langle\left(
                                     \begin{array}{c}
                                       u_1 \\
                                       v_1 \\
                                       p_1 \\
                                       q_1 \\
                                     \end{array}
                                   \right),\left(
                                             \begin{array}{c}
                                               u_2 \\
                                               v_2 \\
                                               p_2 \\
                                               q_2 \\
                                             \end{array}
                                           \right)
\right\rangle_X=\int\limits_0^l\left(E(x)I(x)u_1''(x)\bar u_2''(x)+\rho(x)v_1(x)\bar v_2(x)\right)dx+\varkappa p_1\bar p_2+mq_1\bar q_2
\end{equation*}
and let the differential operator $\tilde{A}:D(\tilde{A})\to X$ with the domain
\begin{equation*}
{\emph D}(\tilde{A})=\left\{\xi=\left(\begin{array}{c}
                 u \\
                 v \\
                 p \\
                 q \\
               \end{array}\right)\in X:\quad
\begin{aligned}&u(x)\in H^4(0,l_0)\cap H^4(l_0,l),\\
&u''(0)=u''(l)=0,\\
&\left.u''\right\rvert_{x=l_0-0}=\left.u''\right\rvert_{x=l_0+0},\\
&v(x)\in\overset{\;\,\circ}H\,\!^2(0,l),\\
&p=u(l_0),\quad q=v(l_0)
\end{aligned}
\right\}\subset X
\end{equation*}
be as follows:
\begin{equation}\label{A}
\tilde{A}:\xi=\left(
               \begin{array}{c}
                 u \\
                 v \\
                 p \\
                 q \\
               \end{array}
             \right)
             \mapsto \tilde{A}\xi=
             \left(
               \begin{array}{c}
                 v \\
                 -\frac1{\rho(x)}\left(E(x)I(x)u''\right)''+\frac1{\rho(x)}\sum\limits_{j=1}^k\psi_j''(x)M_j \\
                 q \\
                 \frac1m(L-\varkappa p+F) \\
               \end{array}
             \right),
\end{equation}
where $y=(M_1,\dots,M_k,F)^T$ is the control and $L=E(x)I(x)\left(\left.u'''\right\rvert_{x=l_0-0}-\left.u'''\right\rvert_{x=l_0+0}\right)$.

A feedback control that ensures the stability of zero equilibrium was proposed in~\cite{ZK2013} as the following functionals:
\begin{equation}\label{Ctrl}
\begin{aligned}
M_j&=-\alpha_j\int\limits_0^l\psi_j''(x)v(x)dx,\quad\alpha_j>0,\quad j=\overline{1,k},\\
F&=-\alpha_0q,\quad\alpha_0>0.
\end{aligned}
\end{equation}

It was shown in~\cite{ZK2013} that the equations of motion~\eqref{Osc}--\eqref{IC} can be represented as differential equation in the operator form:
\begin{equation}\label{infdim}
\frac d{dt}\;\xi(t)=\tilde{A} \xi(t),\quad \xi(t)\in {X}.
\end{equation}

Consider the operator $A$ with the domain ${\emph D}(A)={\emph D}(\tilde{A})$ that acts exactly like $\tilde{A}$ with vanishing controls, i.e. $M_j=0$, $j=\overline{1,k}$, and $F=0$.

In this section, we will construct the inverse operators $A^{-1}$ and $\tilde{A}^{-1}$ and focus on some important features of both operators $A$ and $\tilde A$.

Here and in the sequel,  we will assume that $E$, $I$ and $\rho$ are positive constants for the sake of simplicity .

\begin{lemma} \label{le_1} The operator $\tilde{A}:D(\tilde{A})\to X$ is closed.
\end{lemma}

\begin{proof}
By solving the equation
\begin{equation}\label{InvOpEq}
\tilde{A}\xi=\hat\xi
\end{equation}
with respect to $\xi\in D(A)$ for  $\hat\xi=(\hat u,\hat v,\hat p,\hat q)^T\in X$, we obtain the inverse operator of the following form:
\begin{equation}\label{InvOp}
\tilde{A}^{-1}:\hat\xi=\left(
                        \begin{array}{c}
                          \hat u \vphantom{\begin{aligned}
                                       &\int\limits_0^x\\
                                       &\int\limits_x^l(s-x)
                                            \end{aligned}} \\
                          \hat v  \\
                          \hat p \vphantom{\int\limits_0^l} \\
                          \hat q  \\
                        \end{array}
                      \right)
\mapsto\tilde{A}^{-1}\hat\xi=
                 \left(
   \begin{array}{c}
     \left\{
   \begin{aligned}
    &B_1x+B_2x^3-\frac16\int\limits_0^x(s-x)^3\Gamma(\hat\xi(s))ds,\,x\leq l_0,\\
    &B_3(x-l)+B_4(x-l)^3+\frac16\int\limits_x^l(s-x)^3\Gamma(\hat\xi(s))ds,\,x>l_0
   \end{aligned}
   \right. \\
     \hat u(x) \\
     \frac{EI}\varkappa\left(\int\limits_0^l\Gamma(\hat\xi(s))ds+6(C_2-C_4)\right)-\frac1\varkappa(\alpha_0\hat p+m\hat q) \\
     \hat p \\
   \end{array}
 \right),
\end{equation}
where $\Gamma(\hat\xi(s))=-\frac1{EI}\left(\rho\hat v(x)+\sum\limits_{j=1}^k\psi_j''(x)\int\limits_0^l\psi_j''(x)\hat u(x)dx\right)$, and bounded linear functionals $B_1(\hat\xi),\dots,B_4(\hat\xi)$ are solutions of the algebraic system
\begin{equation*}
K\left(\begin{array}{c}
B_1\\B_2\\B_3\\B_4 \vphantom{\int\limits_0^{l_0}}
\end{array}\right)=-\frac16\left(
                             \begin{array}{c}
                               a_1(\hat\xi) \\
                               3a_2(\hat\xi) \\
                               a_3(\hat\xi) \\
                               6EIa_4(\hat\xi)-\varkappa\int\limits_0^{l_0}(l_0-s)^3\tilde\Gamma(\hat\xi(s))ds-\alpha_0\hat p-m\hat q \\
                             \end{array}
                           \right)
\end{equation*}
with the matrix
\begin{equation*}
K=\left(\begin{matrix}
l_0 & l_0^3 & l-l_0 & (l-l_0)^3 \\
1 & 3l_0^2 & -1 & -3(l-l_0)^2 \\
0 & 1 & 0 & l-l_0 \\
-\varkappa l_0 & 6EI-\varkappa l_0^3 & 0 & -6EI
\end{matrix}\right),
\end{equation*}
and $a_n(\hat\xi)=\int\limits_0^l(l_0-s)^{4-n}\tilde\Gamma(\hat\xi(s))ds$, $n=\overline{1,4}$.
As $$\det K=-EIl(l-l_0+1)-\frac\varkappa3l_0(l-l_0)^2(l-l_0+l_0^2)<0,$$
there exist constants $\bar B_j^n>0$ such that $B_1(\hat\xi),\dots B_4(\hat\xi)$ can be expressed as
$$B_j(\hat\xi)=\sum\limits_{n=0}^3\bar B_j^n\int\limits_0^l(l_0-x)^n\hat v(x)dx+\bar B_j^4\int\limits_0^{l_0}\hat v(x)dx+\bar B_j^5\hat q,\quad j=\overline{1,4}.$$ Note that $\hat B_j^n>0$ are determined only by mechanical parameters of the system considered.

The mapping~\eqref{InvOp} is obviously linear. Let the sequence $\{\hat\xi_n\}_{n=1}^\infty$ be bounded, i.e. $\|\hat\xi_n\|_X\leq\tilde C$. Let us show that the sequence $\{A^{-1}\hat\xi_n\}_{n=1}^\infty$ is bounded in $X$.

As $\hat\xi\in X$, the boundary conditions $\hat u(0)=\hat u(l)=0$ are fulfilled. According to the Poincar{\`e} inequality,
\begin{equation*}
\int\limits_0^l|\hat u|^2dx\leq c\int\limits_0^l\left|\frac{d\hat u}{dx}\right|^2dx,
\end{equation*}
$c>0$. As $\hat u(x)\in\overset{\;\,\circ}H\,\!^2(0,l)$, $H^2(0,l)\subset C^1(0,l)$, according to Lagrange's mean value theorem, there exists a point $\zeta\in(0,l)$ such that
$$\hat u'(\zeta)=\frac{\hat u(l)-\hat u(0)}l=0,$$
so
$$\hat u'=\int\limits_\zeta^x\hat u''(z)dz+\hat u'(\zeta)=\int\limits_\zeta^x\hat u''(z)dz.$$
Applying the Cauchy--Schwarz inequality, we have
\begin{equation*}
\begin{aligned}
\int\limits_0^l|\hat u'(x)|^2dx&=\int\limits_0^\zeta|\hat u'(x)|^2dx+\int\limits_\zeta^l|\hat u'(x)|^2dx=\int\limits_0^\zeta\left|\int\limits_\zeta^x\hat u''(z)dz\right|^2dx+\int\limits_\zeta^l\left|\int\limits_\zeta^x\hat u''(z)dz\right|^2dx\leq\\
&\leq\int\limits_0^\zeta\left((x-\zeta)\int\limits_\zeta^x|\hat u''(z)|^2dz\right)dx+
\int\limits_\zeta^l\left((x-\zeta)\int\limits_\zeta^x|\hat u''(z)|^2dz\right)dx
\overset{(\zeta,x)\subset(0,l)}{\leq}\\
&\leq\frac{l^2}2\int\limits_0^l|\hat u''(x)|^2dx.
\end{aligned}
\end{equation*}

The latter yields the existence of a positive definite quadratic form $\tilde\Lambda=\tilde\Lambda(C_1(\hat\xi),\dots C_4(\hat\xi))$ such that the following estimate is fulfilled:
\begin{equation*}
\|\tilde{A}^{-1}\hat\xi\|^2_X\leq\tilde\Lambda\|\hat\xi\|^2_X;
\end{equation*}
hence, the operator $\tilde{A}^{-1}:X\to X$ is bounded. As $D(\tilde{A}^{-1})=X$, it follows from~\cite[p.~162]{T1980} that the operator $\tilde{A}$ is closed in $X$.
\end{proof}

It will be shown in Theorem~\ref{th_1} that the operator $\tilde{A}$ is maximal in the sense that the range of $(I-\lambda\tilde{A})$ coincides with $X$ for some $\lambda>0$. Besides, the direct substitution leads to
$$\langle\tilde{A}\xi,\bar\xi\rangle_X\leq0,\quad\forall\xi\in D(\tilde{A}),$$
where $\bar\xi$ is the complex conjugate of $\xi$. According to Sobolev's embedding theorems, $H^4\subseteq H^2$, $H^2\subseteq L^2$ (``$\subseteq$'' stands for the compact embedding), which means that the operator $\tilde{A}$ is densely defined in $X$.

Summing up the above, the operator $\tilde{A}$, being densely defined, $m$-dissipative and closed in $X$, satisfies the conditions of the Lumer--Phillips theorem~\cite{LP1961}. Thus, we have
\begin{corollary}
The operator $\tilde{A}:X\to X$ is the infinitesimal generator of a $C_0$-semigroup of operators on $X$.
\end{corollary}

Consequently, the Cauchy problem for~\eqref{infdim} with the initial value from $X$ is well-posed.

\begin{lemma}
The operator $A^{-1}:X\to X$ is compact.
\end{lemma}
\begin{proof}
Exploiting the same technique as in Lemma~\ref{le_1}, one can obtain the following norm estimate for the inverse operator $A^{-1}$ with the domain $D(A^{-1})=X$:
\begin{equation*}
\|A^{-1}\hat\xi\|^2_{X'}\leq\Lambda\|\hat\xi\|^2_X
\end{equation*}
with some positive definite quadratic form $\Lambda$. Here $X'=H^3\times H^1\times {\mathbb C}^2$. So, the operator $A^{-1}:X\mapsto X'$ is closed. According to Sobolev's embedding theorems, $X'\subseteq X$; therefore, the operator $A^{-1}:X\mapsto X$ is compact.
\end{proof}

Note that $A^{-1}$ is well-defined and $D(A^{-1})=X$ (as far as equation~\eqref{InvOpEq} is solvable for any $\hat\xi\in X$). It is easy to verify that the operator $A^{-1}$ is skew-symmetric, i.e.
$$\langle A^{-1}\xi_1,\xi_2\rangle_X=-\langle\xi_1,A^{-1}\xi_2\rangle_X.$$
Thus, the operator $A^{-1}$ satisfies the conditions of the Hilbert--Schmidt theorem, consequently, the eigenvectors of $A^{-1}$ form a basis of $X$. Taking into account the match of eigenvectors of $A$ and $A^{-1}$, we come to the following corollary:
\begin{corollary}\label{cor_2}
Let $\{\xi_n\}_{n=1}^\infty$ be eigenvectors of the operator $A$. Then the system $\{\xi_n\}_{n=1}^\infty$ forms a basis of $X$.
\end{corollary}

In this work, we concentrate on the asymptotic behavior of trajectories of the closed-loop system. The asymptotic behavior is determined by invariant subsets of the set ${Z=\{\xi\in D(A)\,|\,\dot V(\xi)=0\}}$, where $V(\xi)$ is the weak Lyapunov functional constructed in~\cite{ZK2013} in the form
$$2V=\int\limits_0^l(\rho v^2(x)+EI(u''(x))^2)dx+mq^2+\varkappa p^2=\|\xi\|^2_X.$$
The required spectral properties of the infinitesimal generator have been obtained in~\cite{KZB2021}. It has been shown there that the eigenvalues $\lambda_j$, $j=1,2,\dots$, of the corresponding spectral problem can be obtained as the roots of the simplified frequency equation
\begin{equation}\label{Frq}
\Phi_0(\mu)=0,
\end{equation}
where
\begin{equation*}
\Phi_0(\mu)= 2\sin\mu(l-l_0)\sin\mu l_0-\sin\mu l,
\end{equation*}
$\mu=\left(\frac\rho{EI}\omega^2\right)^{1/4}$, $\omega={\rm Im}\lambda$.
In fact, the frequency equation has been obtained in much more complicated form, and its equivalence to~\eqref{Frq} in the sense of limit behaviour of the roots has also been proved in the above-mentioned reference.

\begin{lemma}\label{le_3}
Assume that $\mu_j$ are the roots of the truncated frequency equation~\eqref{Frq} and there are no multiple roots among $\mu_j$. If $\frac{l_0}l$ is rational, then there exists a $\tau>0$ such that the system of functions $\{e^{\lambda_jt}\}_{j=1}^\infty$ is minimal in $L^2(0,\tau)$.
\end{lemma}

\begin{proof}
The function $\Phi_0$ is analytic. As $\Phi_0\not\equiv{\rm const}$, the set
$$\{\mu\in[0,+\infty)\;|\;\Phi_0(\mu)=0\}$$
is totally disconnected.
The condition $\frac{l_0}l=\frac{p_1}{p_2}\in\mathbb{Q}$ guarantees that the function $\Phi_0$ is periodic. Its period $P$ may be calculated according to the following formula:
$$P=\frac{2\pi}{|2l_0-l|}\cdot\frac{|2p_1-p_2|}{GCD(p_2,2p_1-p_2)},\quad p_1,p_2\in\mathbb{N}.$$

If $\mu_0,\dots,\mu_{k-1}$ are roots of the function $\Phi_0$ on $[0;P)$, then
$$\mu_n=\left[\frac nk\right]P+\mu_{\left\{\frac nk\right\}k},\quad n=1,2,\dots$$
As it has been shown in~\cite{KZB2021}, the eigenvalues $\lambda_j=i\sqrt{\frac{EI}\rho}\mu_j^2$ grow quadratically with respect to $j$.
Consider the function $Q'(x)=\max\{n\in\mathbb{N}\;|\;\mu_n^2<x\}$. The following estimates are fulfilled:
$$\frac kP\left(\sqrt x-\mu_{\left\{\frac nk\right\}k}\right)-1\leq Q'(x)\leq\left[\frac kP\left(\sqrt x-\mu_{\left\{\frac nk\right\}k}\right)\right].$$
The quantity of eigenvalues $\lambda_j$ on $[y,y+z)$ is determined by the function $Q[y,y+z)=Q'(y+z)-Q'(y)$ which satisfies the following estimate:
$$Q[y,y+z)\leq\left[\frac kP\sqrt{y+z}\right]-\left[\frac kP(\sqrt y-P)\right]+1\leq\frac kP\left(\sqrt{y+z}-\sqrt y+P\right)+2.$$
Thus, we have
$$\underset{y\to\infty}{\lim\sup}\;\underset{z\to\infty}{\lim\sup}\;\frac{Q[y,y+z)}z=
\underset{y\to\infty}{\lim\sup}\;\underset{z\to\infty}{\lim\sup}\;\frac1z\left(\frac kP(\sqrt{y+z}-\sqrt y+P)+2\right)=0.$$
According to Theorem~1.2.17~\cite{K1992}, the system of functions $\{e^{\lambda_jt}\}_{j=1}^\infty$ is minimal in $L^2(0,\tau)$, $\forall\tau>0$.
\end{proof}

\section{Precompactness of the Trajectories}\label{swc_Precomp}
In this section, we will prove that the trajectories of the closed-loop system are precompact. For this purpose we will construct the resolvent of $\tilde{A}$ and prove its compactness. At first, we will construct the resolvent for the operator $\tilde{A}_M$ with constant parameters on place of controls $M_j$, and then we will substitute  feedback~\eqref{Ctrl} into the obtained equations.

\begin{theorem}\label{th_1}
Let the operator $\tilde{A}$ be defined above. Then the resolvent ${R_\lambda(\tilde{A})=(I-\lambda\tilde{A})^{-1}:X\to X}$ of $\tilde{A}$ is compact.
\end{theorem}
\begin{proof}
Let $M_1,...,M_k$ in~\eqref{A} be constants, $F=-\alpha_0q$.
Denote $\eta=\left(\frac\rho{\lambda^2EI}\right)^{1/4}$, ${\Gamma(\hat\xi(s))=\frac1{EI}\left(\sum\limits_{j=1}^k\psi_j''M_j+\frac\rho\lambda\left(\hat v+\frac{\hat u}\lambda\right)\right)}$, $z_1(x)=\sin{\eta x}\cosh{\eta x}$, $z_2(x)=\cos{\eta x}\sin{\eta x}$, $z_3(x)=z_1(x-s)-z_2(x-s)$, then the equation
\begin{equation}\label{Max}
(I-\lambda\tilde{A}_M)\xi=\hat\xi
\end{equation}
is solvable with respect to $\xi\in D(\tilde{A})$ for any vector $\hat\xi\in X$, and the solution can be written for some $\lambda>0$ as follows:
$$\xi=\left(
               \begin{array}{c}
                 R_1\hat\xi \\
                 R_2\hat\xi \\
                 R_3\hat\xi \\
                 R_4\hat\xi \\
               \end{array}
             \right)\hat\xi,$$
where
\begin{equation}\label{Res}
\begin{aligned}
&R_1\hat\xi=\left\{
   \begin{aligned}
    &B_1z_1(x)+B_2z_2(x)+\frac1{4\eta^3}\int\limits_0^xz_3(x)\Gamma(\hat\xi(s))ds,\,x\leq l_0,\\
    &B_3z_1(x-l)+B_4z_2(x-l)-\frac1{4\eta^3}\int\limits_0^xz_3(x)\Gamma(\hat\xi(s))ds,\,x>l_0,
   \end{aligned}
   \right.\\
&R_2\hat\xi=\left\{\begin{aligned}
    &\left(\frac{EI}\rho\right)^{1/2}\left(2\eta^2(B_1z_1(x)+B_2z_2(x)-\hat u)+\frac1{2\eta}\int\limits_0^xz_3(x)\Gamma(\hat\xi(s))ds\right),\,x\leq l_0,\\
    &\left(\frac{EI}\rho\right)^{1/2}\left(2\eta^2(B_3z_1(x-l)+B_4z_2(x-l)-\hat u)-\frac1{2\eta}\int\limits_0^xz_3(x)\Gamma(\hat\xi(s))ds\right),\,x>l_0,
   \end{aligned}\right.\\
&R_3\hat\xi=B_1z_1(l_0)+B_2z_2(l_0)+\frac1{4\eta^3}\int\limits_0^{l_0}z_3(l_0)\Gamma(\hat\xi(s))ds,\\
&R_4\hat\xi=\left(\frac{EI}\rho\right)^{1/2}\left(2\eta^2(B_1z_1(l_0)+B_2z_2(l_0)-\hat p)+\frac1{2\eta}\int\limits_0^{l_0}z_3(l_0)\Gamma(\hat\xi(s))ds\right).
\end{aligned}
\end{equation}

From the interface conditions $u^{(j)}(l_0-0)-u^{(j)}(l_0+0)=0$, $j=\overline{0,2}$, the values of bounded linear functionals $B_1(\hat\xi),\dots,B_4(\hat\xi)$ can be obtained as solution of the following linear algebraic system:
\begin{equation*}
M\left(
            \begin{array}{c}
              B_1 \vphantom{\int\limits_0^l} \\
              B_2 \vphantom{\int\limits_0^l} \\
              B_3 \vphantom{\int\limits_0^l} \\
              B_4 \vphantom{\int\limits_0^l} \\
            \end{array}
          \right)
=\frac1{4\eta^3}\left(
                  \begin{array}{c}
                    \int\limits_0^lz_3(l_0-s)\Gamma(\hat\xi(s))ds \\
                    2\eta\int\limits_0^l\sin\eta(l_0-s)\sinh\eta(l_0-s)\Gamma(\hat\xi(s))ds \\
                    2\eta^2\int\limits_0^l(z_1(l_0-s)+z_2(l_0-s))\Gamma(\hat\xi(s))ds \\
                    \int\limits_0^lz_3(l_0-s)\Gamma(\hat\xi(s))ds \\
                  \end{array}
                \right)+\gamma,
\end{equation*}
where $\gamma=(0,0,0,\int\limits_0^l\cos\eta(l_0-s)\cosh\eta(l_0-s)\Gamma(\hat\xi(s))ds+\frac{2\eta^2}\rho\left(2\eta^2(m+\alpha_0)\hat p+\sqrt{\frac\rho{EI}}(m\hat q+\hat p)\right))^T$.

The determinant of $M$ for the above system is expanded into Taylor's series with respect to $\eta$ as $\eta\to0$:
\begin{equation*}
\det(M)=\frac{8l}{3m}\left(\frac\rho{EI}\right)^{1/2}\left(\varkappa l_0^2(l-l_0)^2+3EIl\right)\eta^6+O(\eta^{10}).
\end{equation*}
So, it is possible to choose $\eta$ small enough such that $\det(M)\neq0$.

The parameters $M_j$ in~\eqref{Res} can be excluded by substituting $v=R_2\hat\xi$ into the distributed control formula~\eqref{Ctrl}.
The determinant of the resulting algebraic system's matrix is decomposable into Taylor's series with respect to $\mu$ (${\mu=\sqrt[4]{\frac\rho{4\lambda^2EI}}}$):
\begin{equation*}
\det(M_{ij})=1+O(\mu),
\end{equation*}
so we can select $\mu$ small enough such that $\det(M_{ij})\neq0$.
Thus the resolvent of $\tilde{A}$ has been constructed.

Consider the space $X'=\overset{\;\,\circ}H\,\!^3\times H^2\times \mathbb C^2$. Relations~\eqref{Res} define linear continuous mapping $X\mapsto X'$, and there exists a positive  definite quadratic form $\Lambda(B_1(\hat\xi),\dots,B_4(\hat\xi))$ such that the inequality
\begin{equation*}
\|(I-\lambda\tilde{A})^{-1}\hat\xi\|^2_{X'}\leq\Lambda\|\hat\xi\|^2_X
\end{equation*}
is fulfilled for each $\hat\xi\in X$.

As $X'\subseteq X$, the mapping $R_\lambda(\tilde{A}):X\to X$ is a compact operator.
\end{proof}

\begin{remark}
The solvability of equation~\eqref{Max} proves that the operator $\tilde{A}$ is maximal. This fact was taken as an assumption in our previous work on non-asymptotic stability~\cite{ZK2013}.
\end{remark}

\section{Asymptotic Stability}\label{sec_AsStab}

Sufficient conditions for the  asymptotic stability of the considered closed-loop system is formulated in the following theorem.
\begin{theorem}\label{th_2}
Let $\{\xi_i\}_{i\in\mathbb N}$ be eigenvectors of the operator $A$ and, for each $i\in\mathbb N$, either $v_i(l_0)\neq0$ or ${\int\limits_0^l\psi_j''(x)v_i(x)dx\neq0}$ for some $j\in {1,\dots,k}$. Then the solution $\xi=0$ of system~\eqref{infdim} is strongly asymptotically stable.
\end{theorem}

Basically, the assumption $\int\limits_0^l\psi_j''(x)v_i(x)dx\neq0$ for some $j\in\{1,\dots,k\}$ means that the $j$-th piezoactuator is not located at a node of eigenfunction of the beam.

\begin{proof}
It suffices to prove that the set $Z=\{\dot V=0\}$ does not contain any nontrivial trajectory of the closed-loop system.

Let $\xi$ be a solution of~\eqref{infdim}, and let $\xi\in Z$, $\forall t\geq0$. The control $y$ vanishes on the set $Z$. It means that the solution $\xi$ satisfies the following equation in $Z$:
\begin{equation}\label{0Ctlr}
\dot\xi=A\xi.
\end{equation}

Let us expand the vector $\xi$ into the  series of eigenfunctions of $A$:
\begin{equation}\label{xi}
\xi(t)=\sum\limits_{i=1}^\infty r_i(t)\xi_i.
\end{equation}
Here we refer to the system $\{\xi_i\}_{i=1}^\infty$ as a basis of $X$ according to Corollary~\ref{cor_2}.

Substituting the expansion~\eqref{xi} into~\eqref{0Ctlr} and taking into account $A\xi_i=\lambda_i\xi_i$, where $\lambda_i$ are the eigenvalues of $A$, $i=1,2,\dots$, we obtain the following:
\begin{equation*}
\sum\limits_{i=1}^\infty\dot r_i(t)\xi_i=\sum\limits_{i=1}^\infty\lambda_ir_i(t)\xi_i.
\end{equation*}
By virtue of the uniqueness of the expansion in a basis,
\begin{equation*}
\dot r_i=\lambda_ir_i,
\end{equation*}
which leads to $r_i(t)=r_i^0e^{\lambda_it}$.

Insofar as $\xi\in M$,
\begin{equation}\label{1xi}
\int\limits_0^l\psi_j''(x)v(x)dx=0,\quad j=\overline{1,k}.
\end{equation}
Let us write the decomposition of $v(x)$ in the basis $\{\xi_i\}_{i=1}^\infty$:
\begin{equation}\label{2xi}
v(x)=\sum\limits_{i=1}^\infty r_i(t)v_i(x)=\sum\limits_{i=1}^\infty r_i^0e^{\lambda_it}v_i(x).
\end{equation}
Subsituting~\eqref{2xi} into~\eqref{1xi}, we obtain that the linear combination of functions $e^{\lambda_it}$ vanishes. According to Lemma~\ref{le_3}, the functions $e^{\lambda_it}$ are linearly independent, so all the coefficients of linear combination have to be zero:
\begin{equation*}
r_i^0\int\limits_0^l\psi_j''(x)v_i(x)dx=0,\quad i=1,2,\dots
\end{equation*}
Since $\int\limits_0^l\psi_j''(x)v_i(x)dx\neq0$, then $r_i^0=0$, $\forall i\in\mathbb{N}$ which yields $\xi\equiv0$.

So, the set $Z$ does not contain any nontrivial trajectory of system~\eqref{infdim}. According to LaSalle's invariance prinsiple (\cite{L1976}), the trivial solution of the system considered is asymptotically stable.
\end{proof}

\section{Conclusion}\label{sec_Conc}
This work concludes the investigation on asymptotic stability of the vibrating simply-supported flexible beam with distributed control and an attached point mass. The answer to the question about asymptotic stability is positive while the system is not necessarily exponentially stable.

In this paper, we do not take into account the dissipation with the aim to investigate the applicability of the Lyapunov method when there is no natural damping.
The proposed state feedback law is of mathematical interest, while the problems of observer design and observer-based stabilization remain for further study.

This work was supported in part by the German Research Foundation (project ZU 359/2-1).

\end{document}